\documentclass[a4paper,12pt]{amsart}
\usepackage[hmargin=2.5cm,vmargin=2.5cm]{geometry}
\usepackage{hyperref}
\usepackage{graphicx}
\usepackage{bbm}
\hypersetup{colorlinks=true}

\newtheorem{thm}{Theorem}
\newtheorem*{remark}{Remark}
\newtheorem{prop}[thm]{Proposition}
\newtheorem{lemma}[thm]{Lemma}

\renewcommand{\P}{\mathbf P}
\newcommand{\E}{\mathbf E}
\renewcommand{\d}{\mathrm d}
\newcommand{\1}{^{(1)}}
\newcommand{\2}{^{(2)}}
\newcommand{\3}{^{(3)}}
\newcommand{\wt}{\widetilde}
\newcommand{\I}{\mathbbm{1}}
\renewcommand{\O}{\mathcal O}
\DeclareMathOperator*{\PD}{PD}
\allowdisplaybreaks

\begin{document}

\author{R\'eka Szab\'o}
\address{University of Groningen, Nijenborgh 9, 9747 AG Groningen, The Netherlands}
\email{r.szabo@rug.nl}

\author{B\'alint Vet\H o}
\address{MTA--BME Stochastics Research Group, Egry J.\ u.\ 1, 1111 Budapest, Hungary}
\email{vetob@math.bme.hu}

\title{Ages of records in random walks}
\dedicatory{Dedicated to B\'alint T\'oth on the occasion of his $60$th birthday}

\begin{abstract}
We consider random walks with continuous and symmetric step distributions.
We prove universal asymptotics for the average proportion of the age of the $k$th longest lasting record for $k=1,2,\dots$
and for the probability that the record of the $k$th longest age is broken at step $n$.
Due to the relation to the Chinese restaurant process, the ranked sequence of proportions of ages converges to the Poisson--Dirichlet distribution.
\end{abstract}



\maketitle

\section{Introduction}

Recently there has been an increasing interest in the study of record statistics, since this field has found several applications in natural sciences, in particular in finance.
If the discrete time series $x_1, x_2, \dots$ is considered, then we say that a record event happens at time $k$, if $x_k$ is larger than all the previous values in the series.
The record statistics of a time series can be characterized with several variables, such as the time up to which a record survives, the number of records that occur up to time $n$,
the age of the longest lasting record, or the probability that a record event occurs at any given time.
Depending on the exact definition of the model, these variables exhibit interesting behavior.

In this paper we study the record statistics of the random walk where the jumps form an i.i.d. sequence of random variables drawn from a continuous distribution.
A remarkable feature of the record statistics of this model is their universal behaviour,
i.e.\ they do not depend on the details of the jump distribution as long as it is continuous and symmetric \cite{MZ08}.
All these results have their roots in the Sparre Andersen theorem \cite{S53}.

The statistics of the longest lasting record after $n$ steps were first examined in~\cite{MZ08} in the case of random walks with arbitrary symmetric and continuous distribution of the jumps.
The asymptotic expected proportion of the longest lasting record was computed.
Later in~\cite{GMS14} apart from the expected value of the age of the longest lasting record another variable was considered: the probability, that the last record is the longest lasting one.
It was observed that the statistics of the records is sensitive to the exact definition of the length of the last record event.
Three different natural ways were given to define the length of the last record up to time $n$,
and the universal behaviour of the longest lasting record was described in each case.
Some of the constants characterizing the record statistics appeared before in the literature.
The constant in the first case appeared first in~\cite{PY97} as the expected first coordinate of the Poisson--Dirichlet distribution and in~\cite{MZ08} in the context of record ages.
The constant in the second case appeared first in the context of the excursion lengths of a standard Brownian motion \cite{S95}.
These constants also appeared in the study of fractional Brownian motion later in \cite{GMS09, GRS10}. 

The present paper is an extension of the work presented in~\cite{GMS14}.
We consider the full sequence of the $k$th longest lasting records for $k=1,2,\dots$ and by introducing new constants,
we describe their average asymptotic proportions and the limit of the probability that the record of the $k$th longest age is broken at step $n$ for each $k$ as $n\to\infty$.
Some of these constants appeared before in~\cite{S95} and in~\cite{F08} where the probability that the last excursion is the $k$th longest lasting one in a standard Brownian motion was computed.
These results were later extended for general renewal processes \cite{GMS15}.
The corresponding observables for random walk excursions were studied in the recent work \cite{G16}.

The theory of exchangeable partitions was developed in a series of papers by Pitman and coauthors, see also \cite{P02} and references therein.
Based on this theory, the partition generated by the lengths of the record ages in one of the cases studied in the present paper is equal in law to the Chinese restaurant process at time $n$.
As a consequence, the ranked sequence of ages of records rescaled by the total length converges to the Poisson--Dirichlet distribution \cite{P93,PPY92,PY92} with parameter $\alpha=1/2$.
This distribution is also related to the ranked lengths of excursions in diffusion processes \cite{PY97} and to renewal processes \cite{B15}.
The Poisson--Dirichlet distribution with parameter $\alpha=1/2$ was proved to be the limiting distribution of ranked excursion lengths up to time $n$ of the simple symmetric random walk \cite{CH03}.

The paper is organized as follows.
In Section~\ref{s:results}, we state our main result on the asymptotic behaviour of the Laplace transforms of two types of quantities:
the average proportions of the ages of the longest lasting records and the probabilities that the record of the $k$th longest age is broken at step $n$.
In two of the three cases, we can turn these asymptotics into a proof of convergence of the average proportions and in one case to the convergence of the record breaking probabilities by specific arguments.
Then based on the relation to the Chinese restaurant process, we conclude a limit theorem about the rescaled ranked sequence of ages.
The proof of the asymptotics of Laplace transforms is postponed to Section~\ref{s:proof}.

\section{Results}\label{s:results}

Let the sequence $\eta_n$ for $n=1,2,\dots$ consist of i.i.d.\ random variables
from a common symmetric and absolutely continuous distribution.
We consider the random walk $X_0=0$ and $X_n=X_{n-1}+\eta_n$ for $n=1,2,\dots$.
A record event occurs at step $n$, if the random walk takes a value at time $n$ which is larger than all the previous values of the random walk,
that is, $X_n>\max(X_0,\dots,X_{n-1})$ with the convention that the first record occurs at time $0$.

Let $t_1=0$ and recursively define
\begin{equation}
t_k=t_{k-1}+\min\{\tau>0:X_{t_{k-1}+\tau}>X_{t_{k-1}}\}
\end{equation}
for $k=2,3,\dots$ which gives the time when the $k$th record occurs.
The age of the $k$th record is
\begin{equation}
\tau_k=t_{k+1}-t_k
\end{equation}
for $k=1,2,\dots$ which form a sequence of i.i.d.\ random variables by construction.
Let $R_n$ be the number of records until time $n$, i.e.\ define $R_n=m$ if and only if $t_m\le n<t_{m+1}$.
By the i.i.d.\ structure of the ages of records $(\tau_k)_{k=1}^\infty$, $R_n$ becomes a renewal process.
Let $A_n$ denote the current age of the renewal process at time $n$, that is,
$A_n=n-t_m$ if $R_n=m$.
For an illustration, see Figure~\ref{fig:rw}.

\begin{figure}
\begin{center}
\def\svgwidth{250pt}
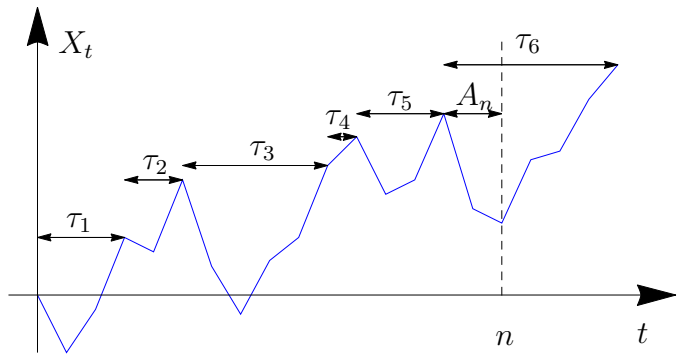\qquad
\caption{A random walk trajectory with $R_n=6$, the age of the last record after $n$ steps is $A_n$, $\tau_6$ or $\tau_5$.\label{fig:rw}}
\end{center}
\end{figure}

There are three different natural choices for the set of ages which one considers.
Given that $R_n=m$, the sequences of ages after $n$ steps in the different cases are given by
\begin{align}
\mathcal{A}\1_{n, m}&=(\tau_1, \tau_2, \dots, \tau_{m-1}, A_n),\label{defA1}\\
\mathcal{A}\2_{n, m}&=(\tau_1, \tau_2, \dots, \tau_{m}),\\
\mathcal{A}\3_{n, m}&=(\tau_1, \tau_2, \dots, \tau_{m-1})
\end{align}
where multiple occurrences are considered with multiplicities.
Now we define the two main quantities of interest.
Given the sequence $\mathcal A^{(\beta)}_{n,m}$ for $\beta=1,2,3$,
let $L_k^{(\beta)}(n)$ denote the $k$th largest element of $\mathcal A^{(\beta)}_{n,m}$ .
Further, let $p_k^{(\beta)}(n)$ denote the probability that the last element of $\mathcal A^{(\beta)}_{n,m}$ is the $k$th longest one of the sequence.
In particular,
\begin{align}
p_1\1(n)&=\P(A_n\ge\max(\tau_1,\tau_2,\dots,\tau_{m-1})),\\
p_1\2(n)&=\P(\tau_{m}\ge\max(\tau_1,\tau_2,\dots,\tau_{m-1})),\\
p_1\3(n)&=\P(\tau_{m-1}\ge\max(\tau_1,\tau_2,\dots,\tau_{m-2})).
\end{align}

On the example of Figure~\ref{fig:rw}, the sequence of the ages of records in the three cases are
$\mathcal A_{16,6}\1=(3,2,5,1,3,2)$, $\mathcal A_{16,6}\2=(3,2,5,1,3,6)$ and $\mathcal A_{16,6}\3=(3,2,5,1,3)$.
The corresponding values of $L_k^{(\beta)}(16)$ are $L_1\1(16)=5$, $L_2\1(16)=3$, $L_3\1(16)=3$, $L_4\1(16)=2$, $L_5\1(16)=2$ and $L_6\1(16)=1$ for $\beta=1$.

We introduce the distribution of the ages of records
\begin{equation}
f(k)=\P(\tau_1=k)=\P(X_k>X_0\ge\max(X_1,\dots,X_{k-1}))
\end{equation}
for $k=1,2,\dots$ and $f(0)=0$, i.e.\ $f(k)$ is the probability that the random walk crosses its starting point between steps $k-1$ and $k$.
Further, let
\begin{equation}
q(k)=\P(\tau_1>k)=\P(\max(X_1,\dots,X_k)<X_0)
\end{equation}
for $k=0,1,\dots$, in particular $q(0)=1$, that is the probability that there is no record until time $k$.
One clearly has
\begin{equation}\label{eq:fk}
f(k)=q(k-1)-q(k)
\end{equation}
by definition.
The distribution of the ages of records is known by the classical Sparre Andersen theorem \cite{S53} which can also be found in Section XII.7 of~\cite{F71}.
\begin{thm}\label{thm:sparre}
Suppose that the step distribution of the random walk $X_n$ is continuous and symmetric.
Then the distribution and the generating function of the ages of records are given by
\begin{align}
q(k)&=\frac1{2^{2k}}\binom{2k}k,& f(k)&=\frac1{2^{2k-1}k}\binom{2k-2}{k-1},\label{defqf}\\
\wt q(z)=\sum_{k\ge0} q(k)z^k&=\frac1{\sqrt{1-z}},& \wt f(z)=\sum_{k\ge1} f(k)z^k&=1-\sqrt{1-z}.\label{qfgenfunction}
\end{align}
\end{thm}
As a consequence, one has the following asymptotics by the Stirling formula:
\begin{equation}\label{qfasymp}
q(k)\sim\frac1{\sqrt{\pi k}},\qquad f(k)\sim\frac1{2\sqrt\pi k^{3/2}}
\end{equation}
where the symbol $\sim$ means that the ratio of the two sides converges to $1$ as $k\to\infty$.

\begin{remark}
By Theorem~\ref{thm:sparre}, the distribution of the ages of records does not depend on the step distribution of the random walk
as long as the latter is continuous and symmetric.
Therefore, instead of the random walk description, the problem of the record age statistics can be posed as a question on renewal processes
where the holding times are distributed according to \eqref{defqf}.

Furthermore, since only the sizes and not the ordering of the renewal intervals matters in what follows,
it is enough to consider the partition of $n$ generated by the ages of records in the random walk up to time $n$.
This partition in the first case which corresponds to \eqref{defA1} equals in law to the Chinese restaurant process at time $n$, see Subsection \ref{ss:chinese}.
\end{remark}

One naturally expects also based on numerical evidence that the asymptotic proportions of the ages of the longest records $\E(L_k^{(\beta)}(n))/n$
and the probabilities $p_k^{(\beta)}(n)$ converge as $n\to\infty$ for fixed $k$ and $\beta$.
We can only prove the convergence of $\E(L_k\1(n))/n$, $\E(L_k\2(n))/n$ and $p_k\1(n)$, but we identify the limits in all the cases under the assumption of convergence.

\subsection{Limiting rank of the last interval and limiting proportions of ages of records}

In order to state the main result of the present paper, we define the constants
\begin{align}
p_k\1&=\frac1{2^{k-1}}\int_0^\infty\frac{x^{-1/2}e^{-x}\Gamma(-1/2,x)^{k-1}}{(x^{-1/2}e^{-x}+\gamma(1/2,x))^k}\,\d x,\label{defp1}\\
p_k\2&=\frac1{2^k}\int_0^\infty\frac{x^{-3/2}(1-e^{-x})\Gamma(-1/2,x)^{k-1}}{(x^{-1/2}e^{-x}+\gamma(1/2,x))^k}\,\d x,\label{defp2}\\
C_k\1&=p_k\1,\label{defC1}\\
C_k\2&=\frac1{2^{k-1}}\int_0^\infty\frac{x^{-1/2}\Gamma(-1/2,x)^{k-1}}{\left(x^{-1/2}e^{-x}+\gamma(1/2,x)\right)^k}\,\d x,\label{defC2}\\
C_k\3&=\frac1{2^k}\int_0^\infty\frac{\Gamma(-1/2,x)^k}{\left(x^{-1/2}e^{-x}+\gamma(1/2,x)\right)^k}\,\d x\label{defC3}
\end{align}
where
\begin{align}
\Gamma(\nu,x)&=\int_x^\infty e^{-t}t^{\nu-1}\,\d t,\\
\gamma(\nu,x)&=\int_0^x e^{-t}t^{\nu-1}\,\d t
\end{align}
are the upper and lower incomplete gamma functions.

\begin{thm}\label{thm:main}
For the Laplace transforms of the probabilities $p_k^{(\beta)}$, for $\beta=1,2$, one has
\begin{equation}\label{pLaplace:asymp}
\sum_{n\ge0}p_k^{(\beta)}(n)e^{-sn}\sim\frac{p_k^{(\beta)}}{s}
\end{equation}
and
\begin{equation}\label{pLaplace:asymp3}
\sum_{n\ge0}p_k\3(n)e^{-sn}\sim\frac1{\sqrt s}\ln\frac1{\sqrt s}
\end{equation}
as $s\to0$.
As a consequence, $p_k\1(n)\to p_k\1$ and if $p_k\2(n)$ converges as $n\to\infty$, then the limit is $p_k\2$.
If $\sqrt n p_k\3(n)/\ln n$ converges, then the limit is $1/(2\sqrt\pi)$, that is
\begin{equation}
p_k\3(n)\sim\frac{\ln n}{2\sqrt{\pi n}}.
\end{equation}

Further, one has the following asymptotics for the Laplace transforms of the expected ages of the longest lasting records
\begin{equation}\label{LLaplace:asymp}
\sum_{n\ge0}\E(L_k^{(\beta)}(n))e^{-sn}\sim\frac{C_k^{(\beta)}}{s^2}
\end{equation}
as $s\to0$ with the constants \eqref{defC1}--\eqref{defC3} for $k=1,2,\dots$ and $\beta=1,2,3$.
Consequently
\begin{equation}\label{L/nconvergence}
\frac{\E(L_k^{(\beta)}(n))}n\to C_k^{(\beta)}
\end{equation}
as $n\to\infty$ for $k=1,2,\dots$ and for $\beta=1,2$.
If the sequence $\E(L_k\3(n))/n$ converges as $n\to\infty$, then the limit is $C_k\3$.
\end{thm}

Table \ref{qkinfty} contains the first few numerical values of $p_k^{(\beta)}$ and $C_k^{(\beta)}$.
The constants $p_k\1$ first appeared in \cite{PY97} as the expected $k$th coordinate of the Poisson--Dirichlet distribution,
$p_k\2$ first appeared in \cite{S95} as the asymptotic probabilities that the last excursion is the $k$th longest lasting one in a standard Brownian motion, see also \cite{F08}.
The constants $C_k\2$ and $C_k\3$ were previously unknown in the literature except for $C_1\3$ which appeared in \cite{GMS14}.
$C_1\2=\infty$ since the expectation of holding times in the corresponding renewal process is infinite, see \eqref{qfasymp}.

\begin{table}
\centering
\begin{tabular}{|c|c|c|c|c|}
\hline
\textbf{k} & $p\1_k=C\1_k$ & $p\2_k$ & $C\2_k$ & $C\3_k$\\
\hline
1 & $0.62651\dots$ & $0.80031\dots$ & $\infty$ & $0.24174\dots$\\
\hline
2 & $0.14301\dots$ & $0.08125\dots$ & $0.18685\dots$ & $0.07999\dots$\\
\hline
3 & $0.06302\dots$ & $0.03342\dots$ & $0.07107\dots$ & $0.04105\dots$\\
\hline
4 & $0.03565\dots$ & $0.01846\dots$ & $0.03826\dots$ & $0.02528\dots$\\
\hline
5 & $0.02300\dots$ & $0.01178\dots$ & $0.02412\dots$ & $0.01724\dots$\\
\hline
6 & $0.01610\dots$ & $0.00819\dots$ & $0.01666\dots$ & $0.01255\dots$\\
\hline
\end{tabular}
\caption{The numerical values of $p\1_k,p\2_k$ and $C^{(\beta)}_k$ for $\beta=1,2,3$ for $1\leq k\leq 6$}\label{qkinfty}
\end{table}

\begin{remark}
The following heuristic argument explains why $\ln n/\sqrt n$ is the order of magnitude of $p_k\3(n)$.
Let $\lambda(n)$ be the random time of the last record until time $n$.
Since the rescaled spent waiting time converges in distribution to the arcsine law (see e.g.\ Section XIV.3 of \cite{F71}),
one has
\begin{equation}\label{lambdadistr}
\P(\lambda(n)=l)\asymp\frac1{\sqrt{ln}}
\end{equation}
for $l\le n/2$ as $n\to\infty$ where $\asymp$ means that the ratio of the two sides are bounded from below and from above by non-trivial constants.
On the other hand, conditionally given that there was a record at time $l$, the number of renewal intervals before $l$ equals $\asymp\sqrt l$,
and these intervals follow each other in a uniform random order.
Hence
\begin{equation}\label{lambdaconditional}
\P\left(\mbox{the last renewal interval in $[0,l]$ is the $k$th longest one}\bigm|\lambda(n)=l\right)\asymp\frac1{\sqrt l}
\end{equation}
for $l$ large enough.
By the law of total probability, if we sum the product of \eqref{lambdadistr} and \eqref{lambdaconditional} for $l$, then the result is asymptotically $\ln n/\sqrt n$.
\end{remark}

\subsection{Relation to the Chinese restaurant process}\label{ss:chinese}

The Chinese restaurant process is a two-parameter stochastic process that can be described as a seating arrangement in a restaurant, see \cite{P95} and \cite{P02}.
Let $\alpha\in[0,1]$ and $\theta>-\alpha$ denote the parameters and consider an initially empty restaurant with an infinite number of circular tables, each with infinite capacity.
The first customer is seated at the first table.
Suppose that $n_1, \dots, n_k$ people are sitting at the tables where $\sum_{i=1}^{k}n_i=n$.
Then the new customer chooses table $i$ for $1\leq i\leq k$ with probability $(n_i-\alpha)/(n+\theta)$, or sits to a new table with probability $(\theta+k\alpha)/(n+\theta)$.

Let $N_i(n)$ denote the $i$th largest element of the random seating arrangement $n_1,\dots,n_k$, that is the size of the $i$th largest table after the arrival of the $n$th customer.
Then the size-biased permutation of the limit of the frequencies
\begin{equation}
\left(\frac{N_1(n)}{n}, \frac{N_2(n)}{n}, \dots \right)
\end{equation} 
can be represented as $(\tau_1, \tau_2, \dots)$,
\begin{equation}
\tau_i=(1-U_i) \prod_{j=1}^{i-1} U_j \qquad i\geq 1
\end{equation}
where the sequence of $U_i$'s is independent with beta distribution of parameter $(\theta + i\alpha, 1-\alpha)$.
This limit is the Poisson--Dirichlet distribution with parameters $(\alpha,\theta)$ described first in~\cite{P90} for $\theta=0$.

\begin{lemma}\label{lemma:PD}
The distribution of the partition generated by the ages of the records in the first case which corresponds to \eqref{defA1}
has the same distribution as the Chinese restaurant process with parameters $(\alpha,\theta)=(1/2,0)$ at time step $n$.
\end{lemma}

\begin{proof}
The equality in law can be seen in two steps.
According to Lemma 7 of \cite{P97}, the Brownian excursion partition up to $n$ has the same distribution as the discrete renewal process with holding time distribution \eqref{defqf}.
The Brownian excursion partition is a partition of integers obtained by the random equivalence relation $i\sim j$ iff
$U_i$ and $U_j$ fall in the same excursion interval of a standard Brownian motion $B$ away from $0$
where $U_i$ are uniformly distributed on $[0,1]$ independent of each other and of $B$.

On the other hand, by Section 4.5 of \cite{P02} based on \cite{PY92} and \cite{P95},
the Brownian excursion partition has the distribution of the Chinese restaurant process with parameters $(\alpha,\theta)=(1/2,0)$.
This finishes the proof.
\end{proof}

\begin{thm}\label{thm:PD}
Let us consider a random walk with arbitrary symmetric and continuous distribution of the jumps.
Let $(\tau_1, \tau_2, \dots, \tau_{m-1}, A_n)$ be the sequence of the ages of its records up to time $n$.
Then the relative sizes of the ranked ages of records jointly converge
\begin{equation}
\left(\frac{L^{(1)}_1(n)}{n}, \frac{L^{(1)}_2(n)}{n}, \dots \right)\stackrel{\d}{\longrightarrow} \PD(1/2, 0)
\end{equation}
in distribution in the vague sense.
\end{thm}

Similar results were already shown for Brownian motion \cite{PY97} and for the simple symmetric random walk \cite{CH03}:
the limiting distribution of ranked excursion lengths up to time $n$ converges in both cases to the Poisson--Dirichlet distribution with parameter $\alpha=1/2$.

\begin{remark}
It follows from Lemma~\ref{lemma:PD} that the size of the age of the last record up to $n$ is a size biased sample from the Chinese restaurant process.
Since the relative cluster sizes in the latter converge to $\PD(1/2,0)$, one can conclude that the probability $p_k\1(n)$ that the last record is the $k$th longest lasting one also converges.
The limit is the expectation of the $k$th coordinate in the $\PD(1/2,0)$ distribution which is exactly $p\1_k$, see Proposition 17 in~\cite{PY97}.
This also proves \eqref{L/nconvergence} for $\beta=1$.
\end{remark}

\subsection{Convergence of average limiting proportions}

The next two propositions imply that the knowledge on the Laplace transforms in Theorem~\ref{thm:main} can be turned into a proof of convergence of $n^{-1}\E(L_k^{(\beta)}(n))$ for $\beta=1,2$.

\begin{prop}\label{prop:cesaro}
For any fixed $n$, the following relation holds:
\begin{equation}
\E(L_k\1(n+1)) = \E(L_k\1(n)) + p_k\1(n),\label{Lprelation}
\end{equation}
or equivalently,
\begin{equation}
\E(L_k\1(n))=\sum_{j=0}^n p_k\1(j)\label{Lpsumrelation}
\end{equation}
with the convention $p_1\1(0)=1$ and $p_k\1(0)=0$ for $k=2,3,\dots$.
As a consequence, the equality of the constants $p\1_k=C\1_k$ for $k=1,2,\dots$ follows without using Theorem~\ref{thm:main}.
Furthermore, the asymptotic equivalence \eqref{pLaplace:asymp} implies the convergence \eqref{L/nconvergence} for $\beta=1$.
\end{prop}

\begin{prop}\label{prop:subadditive}
For any fixed $k$, the sequence $\left(\sum_{j=1}^k \E(L_j\2(n))\right)_n$ is subadditive, that is,
\begin{equation}\label{L2subadditive}
\sum_{j=1}^k \E(L_j\2(n+m))\le\sum_{j=1}^k \E(L_j\2(n))+\sum_{j=1}^k \E(L_j\2(m)).
\end{equation}
Consequently, the proportions $\E(L_k\2(n))/n$ converge as $n\to\infty$.
\end{prop}

\begin{proof}[Proof of Proposition~\ref{prop:cesaro}]
Note that the age of the $k$th longest record can only increase from time $n$ to time $n+1$, if the last interval is the $k$th longest one at time $n$.
This observation means that the difference $L_k(n+1)-L_k(n)$ is a Bernoulli random variable with parameter $p_k\1(n)$ which proves \eqref{Lprelation}, then \eqref{Lpsumrelation} follows immediately.
To prove \eqref{L/nconvergence} for $\beta=1$, the Hardy--Littlewood Tauberian theorem can be used (see e.g.\ Theorem 2 in XIII.5 of \cite{F71}).
The theorem yields that \eqref{pLaplace:asymp} for $\beta=1$ implies that the right-hand side of \eqref{Lpsumrelation} is asymptotically equivalent to $np_k\1$,
i.e.\ \eqref{L/nconvergence} holds for $\beta=1$ with $C\1_k=p\1_k$.
\end{proof}

\begin{proof}[Proof of Proposition~\ref{prop:subadditive}]
We define the shifted random walk $\wt X_l=X_{n+l}-X_n$.
Note that if $r$ is the first record time of the random walk $X_l$ after time $n$, then $r-n$ is also a record time of the shifted random walk $\wt X_l$.
Furthermore, all the record times and ages of records for $X_l$ after time $r$ correspond to those of $\wt X_l$ after time $r-n$.

Let $\mathcal A_n\2$ denote the list of ages of records in the original random walk which started not later than $n$.
Similarly, let $\wt{\mathcal A}_m\2$ be the list of ages of records in the shifted random walk which started not later than $m$.
Now we consider the list $\mathcal A_{n+m}\2$.
All the ages of records which started not later than $n$ are contained in the list $\mathcal A_n\2$.
Ages of records in the original random walk started after time $n$ are contained in $\wt{\mathcal A}_m\2$ by the previous observation on the ages of records after time $r$.
Hence the $k$ largest elements of the list $\mathcal A_{n+m}\2$ can be found either in $\mathcal A_n\2$ or in $\wt{\mathcal A}_m\2$.
Since $\mathcal A_m\2$ and $\wt{\mathcal A}_m\2$ are equal in distribution, for the expectations, the subadditivity relation \eqref{L2subadditive} follows.
By the subadditive lemma, the sequence $\sum_{j=1}^k\E(L\2_j(n))/n$ converges for all $k$ as $n\to\infty$ which finishes the proof of the proposition.
\end{proof}

\subsection{Summation identities}

\begin{prop}\label{prop:sums}
The sum of the constants $p_k^{(\beta)}$ and $C_k^{(\beta)}$ are the following:
\begin{align}
\sum_{k=1}^\infty p_k\1=\sum_{k=1}^\infty p_k\2=\sum_{k=1}^\infty C_k\1&=1,\label{sumC1}\\
\sum_{k=2}^\infty C_k\2&=\frac1{2\sqrt\pi}\int_0^\infty\frac{x^{-1/2}\Gamma(-1/2,x)}{x^{-1/2}+\gamma(1/2,x)}\,\d x\simeq0.43067\dots,\label{sumC2}\\
\sum_{k=1}^\infty C_k\3&=\frac12.\label{sumC3}
\end{align}
\end{prop}

The proof of the summation identities \eqref{sumC1}--\eqref{sumC3} in Proposition~\ref{prop:sums} are straightforward
by summing the geometric series under the integral sign in \eqref{defp1}--\eqref{defC3} and by the integration by parts formula
\begin{equation}\label{intparts}\begin{aligned}
\Gamma(-1/2,x)&=2x^{-1/2}e^{-x}-2\Gamma(1/2,x)\\
&=2x^{-1/2}e^{-x}+2\gamma(1/2,x)-2\sqrt\pi
\end{aligned}\end{equation}
where the fact $\Gamma(1/2)=\sqrt\pi$ is used.

Next we explain why it is natural to expect the identities \eqref{sumC1} and \eqref{sumC3} to hold.
Clearly, for any finite $n$,
\[\sum_{k=1}^\infty p_k\1(n)=\sum_{k=1}^\infty p_k\2(n)=\sum_{k=1}^\infty\frac{\E(L_k\1(n))}n=1\]
holds, hence by Fatou's lemma, the sums in \eqref{sumC1} are at most $1$.
The reason for these sums not being less than $1$ is the following.
One expects an invariance principle for the ages of the records of the random walk $X_n$ to hold, i.e.\ ages of records of the walk should converge to the ages of records of Brownian motion.
Let us consider the trajectories of the random walk $X_n$ as continuous functions by linear interpolation.
By reflecting a trajectory of $X_n$ downwards off $0$ in the Skorokhod sense (see e.g.\ Section VI.2 in \cite{RY99}),
the record ages of the original trajectory correspond to the excursion lengths of the reflected trajectory (up to an error which is at most $1$).
Invariance principle for the excursions of the random walk was proved in \cite{CH04} for the simple random walk, but it is expected for other step distributions as well.
Provided that such an invariance principle is proved for the continuous and symmetric step distributions which are considered in the present paper,
then the equality in \eqref{sumC1} readily follows.

To explain \eqref{sumC3}, one should observe that
\begin{equation}\label{n/2}
\sum_{k=1}^\infty \E(L_k\3(n))=\frac n2
\end{equation}
holds for any $n$ since the left-hand side is the expected time of the last record before $n$ which is the expected time of the maximal value of the random walk $X_k$ for $0\le k\le n$.
By reverting the trajectory, the time of the maximal value of the random walk $\widehat X_k=X_{n-k}-X_n$ for $0\le k\le n$ and that of the original walk add up to $n$.
Since this time reversal preserves the measure, it explains \eqref{n/2}.
The same invariance principle for the ages of records of the random walk $X_n$ which was used in the explanation of \eqref{sumC1} could turn this argument into a proof of \eqref{sumC3}.

\section{Proof of Laplace transform asymptotics}\label{s:proof}

This section is devoted to the proof of Theorem~\ref{thm:main} which is the main result of the paper.
We start with an a priori bound on the constants $C_k^{(\beta)}$.

\begin{lemma}\label{lemma:Ckbound}
For any $\beta=1,2,3$ and for $k\ge2$, we have
\begin{equation}\label{Lkbound}
\frac{L_k^{(\beta)}(n)}n\le\frac1{k-1}
\end{equation}
almost surely.
As a consequence, we have
\begin{equation}\label{Ckbound}
C_k^{(\beta)}\le\frac1{k-1}.
\end{equation}
\end{lemma}

\begin{proof}
The bound \eqref{Lkbound} is clear, since at least $k-1$ of the $k$ longest intervals are in $[0,n]$.
Consequently,
\begin{equation}
\sum_{n\ge0}\E(L_k^{(\beta)}(n))e^{-sn}\le\frac1{k-1}\sum_{n\ge0} ne^{-sn}=\frac1{k-1}\frac{e^{-s}}{(1-e^{-s})^2}\sim\frac1{k-1}\frac1{s^2}
\end{equation}
which proves \eqref{Ckbound}.
\end{proof}

\begin{proof}[Proof of Theorem~\ref{thm:main}]
To compute $L^{(\beta)}_k(n)$ and $p^{(\beta)}_k(n)$, first we have to determine the joint distribution of $R_n$ and the sequence of ages up to time $n$.
Assume that $R_n=m$ and let $\bar{l}\1=(l_1, l_2, \dots, l_{m-1}, a)$, $\bar{l}\2=(l_1, l_2, \dots, l_m)$ and $\bar{l}\3=(l_1, l_2, \dots, l_{m-1})$
be the possible realizations of ages in the three different cases.
The probability of a realization is
\begin{equation}\label{eq:vsz}
\P(\bar{l}^{(\beta)}, n, m)=\P(\mathcal{A}^{(\beta)}_{n, m}=\bar{l}^{(\beta)}, R_n=m)
\end{equation}
which can be expressed with $q(k)$ and $f(k)$ of \eqref{defqf} in the cases $\beta=1,2,3$ as
\begin{align}
\P(\bar{l}\1, n, m) &= f(l_1) f(l_2) \dots f(l_{m-1}) q(a) \I\left(\sum_{k=1}^{m-1} l_k+a=n\right)\label{eq:vsz2}\\
\P(\bar{l}\2, n, m) &= f(l_1) f(l_2) \dots f(l_{m}) \I\left(\sum_{k=1}^{m-1} l_k < n \leq \sum_{k=1}^{m} l_k\right)\label{eq:vsz2II}\\
\P(\bar{l}\3, n, m) &= f(l_1) f(l_2) \dots f(l_{m-1}) \sum_{a\geq 0} q(a) \I\left(\sum_{k=1}^{m-1} l_k+a=n\right)\label{eq:vsz2III}
\end{align}
where $\I(\cdot)$ is the indicator function.
This explicit expression for the joint distribution in case $\beta=1$ first appeared in \cite{MZ08}.

\paragraph{Case $\beta=1$.}
Let $p\1_k(n, m)$ denote the probability that $R_n=m$ and the last element of $\mathcal A\1_{n,m}$ is the $k$st longest one of the sequence.
In this case, there are exactly $k-1$ longer intervals than $A_n$ among $\tau_1, \dots, \tau_{m-1}$.
This probability can be expressed as
\begin{equation}\begin{aligned}
p\1_k(n, m)&=\P(A_n=L\1_k(n), R_n=m)\\
&=\binom{m-1}{k-1} \sum_{a\geq 0} \sum_{l_1=a+1}^{\infty} \dots \sum_{l_{k-1}=a+1}^{\infty} \sum_{l_k=1}^{a} \dots \sum_{l_{m-1}=1}^{a} \P(\bar{l}\1, n, m).
\end{aligned}\end{equation}
By summing over all $m$, we get exactly $p\1_k(n)$, the probability that the last element of $\mathcal A\1_{n,m}$ is the $k$st longest one.
For its generating function,
\begin{equation}\label{p1genfunction}\begin{aligned}
\wt{p}\1_k(z) &= \sum_{n\geq 0} p\1_k(n) z^n\\
&= \sum_{n\geq 0} \sum_{m\geq 1} p\1_k(n, m) z^n\\
&= \sum_{m\geq 1}\sum_{j\geq 0} q(j) z^j \binom{m-1}{k-1} \left(\sum_{l=1}^{j}f(l) z^l\right)^{m-k} \left(\sum_{l=j+1}^{\infty} f(l) z^l\right)^{k-1}.
\end{aligned}\end{equation}

Next we apply the identity
\begin{equation}\label{binomsummation}
\sum_{m\geq 1} \binom{m-k}{k-1} x^{m-k}=\frac 1 {(1-x)^k}.
\end{equation}
The expression in \eqref{p1genfunction} for the generating function can be written as
\begin{equation}
\wt{p}\1_k(z) = \sum_{j\geq 0} \frac{q(j) z^j \left(\sum_{l=j+1}^{\infty} f(l) z^l\right)^{k-1}}{\left(1-\sum_{l=1}^{j}f(l) z^l\right)^k}.
\end{equation}
The following identity can be derived from \eqref{eq:fk}
\begin{equation}\label{eq:fqosszef}
1-\sum_{k=1}^{t} f(k) z^k = q(t) z^t + (1-z) \sum_{k=0}^{t-1} q(k) z^k.
\end{equation}
Using this identity, the denominator can be rewritten as
\begin{equation}
\wt{p}\1_k(z) = \sum_{j\geq 0} \frac{q(j) z^j \left(\sum_{l=j+1}^{\infty} f(l) z^l\right)^{k-1}}{\left(q(j) z^j + (1-z) \sum_{l=0}^{j-1} q(l) z^l\right)^k}.
\end{equation}

To prove \eqref{pLaplace:asymp} for $\beta=1$, one has to show
\begin{equation}\label{Laplaceconv}
s\wt p\1_k(e^{-s})= s\sum_{j\geq 0} \frac{q(j) e^{-sj} \left(\sum_{l=j+1}^{\infty} f(l) e^{-sl}\right)^{k-1}}{\left(q(j) e^{-sj} + (1-e^{-s}) \sum_{l=0}^{j-1} q(l) e^{-sl}\right)^k}\to p_k\1
\end{equation}
for any $k$ fixed as $s\to0$.
Using \eqref{qfasymp}, one can see that there is a $C>0$ such that for any $x>0$ large enough, the summand $j=x/s$ in \eqref{Laplaceconv} is bounded by $Ce^{-x}$ uniformly in $s$.
This bound can then be turned into a uniform exponential bound on the tail of the sum in \eqref{Laplaceconv}.
The uniform exponential bound justifies the exchange of limits and reduces \eqref{Laplaceconv} to
\begin{equation}\label{Laplace:doublelimit}
\lim_{K\to\infty}\lim_{s\to0}\sum_{j=0}^{K/s} \frac{sq(j) e^{-sj} \left(\sum_{l=j+1}^{\infty} f(l) e^{-sl}\right)^{k-1}}{\left(q(j) e^{-sj} + (1-e^{-s}) \sum_{l=0}^{j-1} q(l) e^{-sl}\right)^k}= p_k\1.
\end{equation}
where sum on the left-hand side is a Riemann sum for the integral in \eqref{defp1} on $[0,K]$ instead of $[0,\infty)$.
This verifies \eqref{pLaplace:asymp} for $\beta=1$.
The convergence $p_k\1(n)\to p_k\1$ follows from the remark after Theorem~\ref{thm:PD}.

For the proof of \eqref{LLaplace:asymp}, one can use Proposition~\ref{prop:cesaro} to get the relation
\begin{equation}
(1-z) \sum_{n\geq 0} \E(L\1_k(n)) z^n = z \sum_{n\geq 0} p\1_k(n) z^n.
\end{equation}
By substituting $z=e^{-s}$, one readily gets \eqref{LLaplace:asymp} for $\beta=1$ with $C_k\1=p_k\1$.
The convergence \eqref{L/nconvergence} for $\beta=1$ follows as a consequence of Proposition~\ref{prop:cesaro} and the Hardy--Littlewood Tauberian theorem.

\paragraph{Case $\beta=2$.}
With similar notations as in the previous case, we have the following formula
\begin{align*}
p\2_k(n, m)&=\P(\tau_m=L\2_k(n), R_n=m)\\
&=\binom{m-1}{k-1} \sum_{j\geq 1} \sum_{l_1=j+1}^{\infty} \dots \sum_{l_{k-1}=j+1}^{\infty} \sum_{l_k=1}^{j} \dots \sum_{l_{m-1}=1}^{j} \P(\bar{l}\2, n, m).
\end{align*}
Using \eqref{eq:vsz2II} and \eqref{binomsummation}, the generating function of the sequence $p_k\2(n)$ can be expressed as
\begin{align*}
\wt{p}\2_k(z) &= \sum_{m\geq 1}\sum_{j\geq 0} f(j) \frac{1-z^j}{1-z} \binom{m-1}{k-1} \left(\sum_{l=1}^{j}f(l) z^l\right)^{m-k} \left(\sum_{l=j+1}^{\infty} f(l) z^l\right)^{k-1}\\
&= \sum_{j\geq 0} \frac 1 {1-z} \left(\frac{f(j) (1-z^j) \left(\sum_{l=j+1}^{\infty} f(l) z^l\right)^{k-1}}{\left(1-\sum_{l=1}^{j}f(l) z^l\right)^k}\right).
\end{align*}
Following the method presented in the previous case, we similarly get \eqref{pLaplace:asymp} for $\beta=2$ with $p_k\2$ given in \eqref{defp2}.

However in the $\beta=2$ case, there is no simple relation of $\E(L_k^{(\beta)}(n))$ and $p_k^{(\beta)}(n)$ like \eqref{Lprelation} for $\beta=1$,
hence we cannot deduce \eqref{LLaplace:asymp} from \eqref{pLaplace:asymp}, but we proceed with direct computations.
By \eqref{qfasymp}, the expected total length of the last unfinished interval is infinite, therefore for $\beta=2$, one has $\E(L\2_1(n))=\infty$.
In order to compute the expected interval lengths for $k>1$, denote the distribution function of $L\2_k(n)$ by $F\2_k(t, n)$.
By the law of total probability, one can write
\begin{equation}
F\2_k(t, n)=\P(L\2_k(n)\leq t) = F\2_{k-1}(t, n) + \P(L\2_k(n)\leq t, L\2_{k-1}(n)> t).
\end{equation}
We decompose the distribution function depending on the number of renewal intervals until $n$ as
\begin{equation}\label{Fk2decomp}
F\2_k(t,n)=\sum_{m\geq 1} F\2_k(t, n, m)
\end{equation}
where
\begin{align*}
F\2_k(t, n, m) &= \P(L\2_k(n) \leq t, R_n=m)\\
&=F\2_{k-1}(t, n, m) + \P(L\2_k(n)\leq t, L\2_{k-1}(n) > t, R_n=m).
\end{align*}

In the last probability above, if the first $k-1$ longest lasting time intervals are greater than $t$, we can distinguish two different cases:
if these intervals are all among $\tau_1, \dots, \tau_{m-1}$, or if $\tau_m$ is one of these intervals.
Based on this observation, $F\2_k(t, n, m)$ can be rewritten as follows
\begin{equation}\label{Fk2mrecursion}\begin{aligned}
F\2_k(t, n, m)&= F\2_{k-1}(t, n, m)\\
&\qquad+ \binom{m-1}{k-1} \sum_{l_1=t+1}^{\infty} \dots \sum_{l_{k-1}=t+1}^{\infty} \sum_{l_{k}=1}^{t} \dots \sum_{l_{m-1}=1}^{t} \sum_{l_m=0}^{t} \P(\bar{l}\2, n, m)\\
&\qquad+ \binom{m-1}{k-2} \sum_{l_1=t+1}^{\infty} \dots \sum_{l_{k-2}=t+1}^{\infty} \sum_{l_{k-1}=1}^{t} \dots \sum_{l_{m-1}=1}^{t} \sum_{l_m=t+1}^{\infty} \P(\bar{l}\2, n, m).
\end{aligned}\end{equation}
The generating function
\begin{equation}
\wt{F}\2_k(t, z)= \sum_{n\geq 0} F\2_k(t, n) z^n
\end{equation}
can be decomposed based on \eqref{Fk2decomp}, and using the recursion \eqref{Fk2mrecursion} for each term and by substituting the value of $\P(\bar{l}\2, n, m)$ from \eqref{eq:vsz2II}, one obtains
\begin{equation}\label{eq:fkII}
\begin{aligned}
\wt{F}\2_k(t, z)= \wt{F}\2_{k-1}(t, z) &+ \frac 1 {1-z}\left(\frac{\left(\sum_{j=t+1}^{\infty}f(j)z^j\right)^{k-1}\sum_{j=1}^{t} f(j) (1-z^j)}{\left(1-\sum_{j=1}^{t} f(j) z^j\right)^k}\right)\\
&+\frac 1 {1-z} \left(\frac{\left(\sum_{j=t+1}^{\infty} f(j) z^j\right)^{k-2} \sum_{j=t+1}^{\infty} f(j) (1-z^j)}{\left(1-\sum_{j=1}^{t} f(j) z^j\right)^{k-1}}\right).
\end{aligned}
\end{equation}
By computing the expected value using the formula
\begin{equation}\label{expectedvalue}
\E(L\2_k(n)) = \sum_{t\geq 0} \left(1-F\2_k(t, n)\right)
\end{equation}
and by considering the generating functions of the two sides of \eqref{expectedvalue}, one can turn the recursion of generating functions \eqref{eq:fkII} into the asymptotic recursion of Laplace transforms
\begin{multline}
\sum_{n\geq 0} e^{-sn} \E(L\2_{k-1}(n)) \sim  \sum_{n\geq 0} e^{-sn} \E(L\2_k(n))\\
+\frac 1 {s^2}\int_0^{\infty} \left(\frac 1 {2^{k-1}} \frac{\int_x^{\infty} y^{-3/2} (1-e^{-y}) \,\d y \left(\Gamma(-1/2,x)\right)^{k-2}}{\left(x^{-1/2} e^{-x} + \gamma(1/2,x)\right)^{k-1}}\right.\\
+ \left. \frac 1 {2^k} \frac{ \int_0^{x} y^{-3/2} (1-e^{-y}) \,\d y \left(\Gamma(-1/2,x)\right)^{k-1}}{\left(x^{-1/2} e^{-x} + \gamma(1/2,x)\right)^k}\right) \,\d x
\end{multline}
as $s\to0$ for $k\ge3$ by following the method presented for $\beta=1$.
Simplifying the integral above and using the integration by parts \eqref{intparts}, one gets
\begin{equation}
\sum_{n\geq 0} e^{-sn} \E(L\2_{k-1}(n))
\sim\sum_{n\geq 0} e^{-sn} \E(L\2_k(n))+\frac 1 {s^2}\frac{\sqrt\pi}{2^{k-2}}\int_0^\infty\frac{x^{-1/2}\Gamma(-1/2,x)^{k-2}}{(x^{-1/2} e^{-x} + \gamma(1/2,x))^k}\,\d x
\end{equation}
for $k\ge3$.
Having a $k\to\infty$ decay bound on $L\2_k(n)$ by Lemma~\ref{lemma:Ckbound}, the asymptotic equivalence \eqref{LLaplace:asymp} follows for some sequence of constants $C_k\2$ which satisfies the recursion
\begin{equation}
C_{k-1}\2=C_k\2+\frac{\sqrt\pi}{2^{k-2}}\int_0^\infty\frac{x^{-1/2}\Gamma(-1/2,x)^{k-2}}{(x^{-1/2} e^{-x} + \gamma(1/2,x))^k}\,\d x
\end{equation}
for $k\ge3$.
The solution for this recursion with $\lim_{k\to\infty}C_k\2=0$ is exactly \eqref{defC2}.
Hence \eqref{LLaplace:asymp} follows for $\beta=2$.
The sequence $\E(L\2_k(n))/n$ converge as $n\to\infty$ by Proposition~\ref{prop:subadditive},
thus \eqref{LLaplace:asymp} together with the Hardy--Littlewood Tauberian theorem implies \eqref{L/nconvergence} for $\beta=2$.

\paragraph{Case $\beta=3$.}
With the same notations as before, we have
\begin{align*}
p\3_k(n, m)&=\P(\tau_m=L\3_k(n), R_n=m)\\
&=\binom{m-1}{k-1} \sum_{j\geq 1} \sum_{l_1=j+1}^{\infty} \dots \sum_{l_{k-1}=j+1}^{\infty} \sum_{l_k=1}^{j} \dots \sum_{l_{m-2}=1}^{j} \sum_{a=0}^{\infty} \P(\bar{l}\3, n, m)
\end{align*}
and by using \eqref{binomsummation} for the generating function,
\begin{equation}\label{p3tilde}
\wt{p}\3_k(z) = \sum_{n\ge0} p_k\3(n)z^n = \wt{q}(z) \sum_{j\geq 0} \frac{ f(j) z^j \left(\sum_{l=j+1}^{\infty} f(l) z^l\right)^{k-1}}{\left(1-\sum_{l=1}^{j}f(l) z^l\right)^k}
\end{equation}
where $\wt q(z)$ is given by \eqref{qfgenfunction}.
Again by \eqref{qfgenfunction}, we rewrite the denominator of \eqref{p3tilde} as
\begin{equation}\label{fsumidentity}
1-\sum_{l=1}^{j}f(l) z^l=\sqrt{1-z}+\sum_{l=j+1}^\infty f(l) z^l.
\end{equation}
Since $\wt q(e^{-s})\sim 1/\sqrt s$, it is enough for the proof of \eqref{pLaplace:asymp3} to show that
\begin{equation}\label{p3asymp}
\sum_{j\geq 0} \frac{ f(j) e^{-sj} \left(\sum_{l=j+1}^{\infty} f(l) e^{-sl}\right)^{k-1}}{\left(\sqrt{1-e^{-s}}-\sum_{l=j+1}^\infty f(l) e^{-sl}\right)^k}\sim\ln\frac1{\sqrt s}
\end{equation}
as $s\to0$.

To verify \eqref{p3asymp}, we fix an $\varepsilon>0$ small and we divide the sum into a sum for $j\le\varepsilon/s$ and a sum for $j>\varepsilon/s$.
For the sum for $j>\varepsilon/s$, we can use the argument which lead to \eqref{Laplace:doublelimit} for the $\beta=1$ case to derive a uniform exponential tail bound.
Hence the $j>\varepsilon/s$ part of the sum in \eqref{p3asymp} can be considered as a Riemann sum for the integral
\begin{equation}\label{intepsiloninfty}
\int_\varepsilon^\infty\frac{x^{-3/2}\Gamma(-1/2,x)^{k-1}}{(2\sqrt\pi+\Gamma(-1/2,x))^k}\,\d x.
\end{equation}
Since $\Gamma(-1/2,x)=(1+\O(x))2/\sqrt x$ for $x\to0$, the integrand above is equal to $(1+\O(\sqrt x))/(2x)$ for small $x$.
Therefore, the integral in \eqref{intepsiloninfty} is asymptotically $-\frac12\ln\varepsilon+\O(1)$ as $\varepsilon\to0$ where the error terms are uniform.

For the $j\le\varepsilon/s$ part of the sum, we again rewrite the left-hand side of \eqref{p3asymp} using \eqref{fsumidentity} and we consider the sum
\begin{equation}\label{p3asymplower}
\sum_{j=0}^{\varepsilon/s} \frac{ f(j) e^{-sj} \left(1-\sqrt{1-e^{-s}}-\sum_{l=1}^j f(l) e^{-sl}\right)^{k-1}}{\left(1-\sum_{l=1}^j f(l) e^{-sl}\right)^k}.
\end{equation}
Next we compare the two series $\sum_{l=1}^\infty f(l)=1$ and $\sum_{l=1}^\infty f(l)e^{-sl}=1-\sqrt{1-e^{-s}}$.
Note that both series have positive entries, and the second one is obtained from the first one by multiplying the summands by numbers from $(0,1)$.
Since the two sums differ by $\sqrt{1-e^{-s}}=\O(\sqrt s)$, this remains an upper bound on the difference of their partial sums, i.e.
\begin{equation}
\left|\sum_{l=1}^j f(l)e^{-sl}-\sum_{l=1}^j f(l)\right|\sim\left|\sum_{l=1}^j f(l)e^{-sl}-\left(1-\frac1{\sqrt{\pi j}}\right)\right|=\O(\sqrt s).
\end{equation}
Using this bound, the sum \eqref{p3asymplower} can be written as
\begin{equation}\label{p3asymplowerwitherror}
\sum_{j=0}^{\varepsilon/s} \frac{\frac1{2\sqrt\pi j^{3/2}}(1+\O(\varepsilon))\left(\frac1{\sqrt{\pi j}}+\O(\sqrt s)\right)^{k-1}}{\left(\frac1{\sqrt{\pi j}}+\O(\sqrt s)\right)^k}
\end{equation}
where the $\O(\varepsilon)$ and $\O(\sqrt s)$ errors are uniform.
Since $j\le\varepsilon/s$, we have $1/\sqrt j\ge\sqrt{s/\varepsilon}$,
hence the $\O(\sqrt s)$ errors are negligible in \eqref{p3asymplowerwitherror} if $\varepsilon$ is small enough.
Therefore, \eqref{p3asymplowerwitherror} is asymptotically equal to
\begin{equation}\label{harmonicseries}
\frac12(1+\O(\varepsilon))\sum_{j=1}^{\varepsilon/s}\frac1j\sim\frac12(1+\O(\varepsilon))(\ln\varepsilon-\ln s).
\end{equation}

We have shown that as $s\to0$, the left-hand side of \eqref{p3asymp} is asymptotically equal to the sum of the integral \eqref{intepsiloninfty} and \eqref{harmonicseries},
i.e.\ it is equal to $-\frac12(1+\O(\varepsilon))\ln s+\O(\varepsilon\ln\varepsilon)$.
Since $\varepsilon$ was arbitrary, \eqref{pLaplace:asymp3} is proved.

Following the same computation as in the $\beta=2$ case, the distribution function of $L\3_k(n)$ restricted to the event $R_n=m$ is
\begin{align*}
F\3_k(t, n, m) &= F\3_{k-1}(t, n, m) + \P(L\3_k(n)\leq t, L\3_{k-1}(n) > t, R_n=m)\\
&= F\3_{k-1}(t, n, m) + \binom{m-1}{k-1}
\sum_{l_1=t+1}^{\infty} \dots \sum_{l_{k-1}=t+1}^{\infty} \sum_{l_{k}=1}^{t} \dots \sum_{l_{m-1}=1}^{t} \sum_{a=0}^{\infty} \P(\bar{l}\3, n, m)
\end{align*}
and the generating function of $F_k\3(t,n)=\sum_{m\ge1}F_k\3(t,n,m)$ is
\begin{equation}\label{eq:fkIII}\begin{aligned}
\wt{F}\3_k(t, z)&= \sum_{n\geq 0} F\3_k(t, n) z^n\\
&=\sum_{n\geq 0} \sum_{m\geq 1} F\3_k(t, n, m) z^n\\
&= \wt{F}\3_{k-1}(t, n) + \wt{q}(z) \frac{\left(\sum_{j=t+1}^{\infty}f(j)z^j\right)^{k-1}}{\left(1-\sum_{j=1}^{t} f(j) z^j\right)^k}.
\end{aligned}\end{equation}

The generating function of $\E(L\3_k(n))$ can easily be expressed from the previous equation.
For $k\geq 2$, we get the asymptotic recursion
\begin{equation}
\sum_{n\geq 0}e^{-sn} \E(L\3_k(n)) \sim  \sum_{n\geq 0} e^{-sn} \E(L\3_{k-1}(n))
-\frac 1 {s^2}\int_0^{\infty} \frac{\sqrt\pi}{2^{k-1}} \frac{\Gamma(-1/2,x)^{k-1}}{\left(x^{-1/2} e^{-x} + \gamma(1/2,x)\right)^k} \,\d x.
\end{equation}
As for the case $\beta=2$, it readily implies a recursion for $C_k\3$ which is satisfied by the right-hand side \eqref{defC3}.
Further the limit of $C_k\3$ as $k\to\infty$ is $0$ by Lemma~\ref{lemma:Ckbound}, which is shown by a dominated convergence argument for the right-hand side of \eqref{defC3}.
This finishes the proof.
\end{proof}

\section{Conclusion}

The asymptotics for the Laplace transform of the probability that the last record is the $k$th longest lasting one
and for the Laplace transform of the expected proportion of the $k$th longest lasting record is understood by Theorem~\ref{thm:main}.
In some of the cases, these could rigorously be turned into the asymptotics for the sequences themselves by specific arguments.
The rigorous proof of the asymptotics in other cases however remains open.

\section*{Acknowledgements}
The authors thank Bal\'azs R\'ath for discussions related to this paper, in particular for the remark after Theorem~\ref{thm:main}.
They are grateful for the anonymous referees for their insightful suggestions and comments.
This work was supported by OTKA (Hungarian National Research Fund) grant K100473.
B.\ V.\ is grateful for the Postdoctoral Fellowship of the Hungarian Academy of Sciences and for the Bolyai Research Scholarship.

\end{document}